\documentclass[a4paper, 11pt]{article}

\usepackage[latin1]{inputenc}
\usepackage[T1]{fontenc}
\usepackage[francais]{babel}
\usepackage{amssymb}
\usepackage{amsmath}
\usepackage{amsthm}
\usepackage{amscd}
\usepackage{amsfonts}
\usepackage{stmaryrd}
\usepackage{pb-diagram}
\usepackage{epic,eepic,epsfig}
\usepackage{a4wide}
\usepackage{nextpage}
\usepackage{fancyhdr}
\newtheorem{prop}{Proposition}[section]
\newtheorem{cor}[prop]{Corollaire}

\newtheorem{thm}[prop]{Th{\'e}or{\`e}me}

\newtheorem{ques}[prop]{Question}
\newtheorem{defin}[prop]{D\'efinition}

\newtheorem{rem}[prop]{Remarque}

\newcommand{\Jac} {\mathop{\mathrm{Jac}}}
\newcommand{\Ima} {\mathop{\mathrm{Im}}}

\newcommand{\disc} {\mathop{\mathrm{disc}}}

\newcommand{\ordv} {\mathop{\mathrm{ord}_{v}}}
\newcommand{\Ree} {\mathop{\mathrm{Re}}}
\newcommand{\Card} {\mathop{\mathrm{Card}}}
\newcommand{\Sp} {\mathop{\mathrm{Sp}}}

\newcommand{\divi} {\mathop{\mathrm{div}}}
\newcommand{\pgcd} {\mathop{\mathrm{pgcd}}}
\newcommand{\degr} {\mathop{\mathrm{deg}}}
\newcommand{\degrar} {\mathop{\widehat{\mathrm{deg}}}}

\newcommand{\Spec} {\mathop{\mathrm{Spec}}}

\newcommand{\supp} {\mathop{\mathrm{supp}}}

\newcommand{\Nk} {\mathop{\mathrm{N}_{k/\mathbb{Q}}}}

\newcommand{\Lie} {\mathop{\mathrm{Lie}}}
\newcommand{\Pic} {\mathop{\mathrm{Pic}}}

\newcommand{\hF} {\mathop{h_{\mathrm{F}}}}
\newcommand{\hFplus} {\mathop{h_{\mathrm{F}^+}}}

\begin{document}

\title{D\'ecompositions en hauteurs locales}
\author{Fabien Pazuki}

\maketitle

\begin{centering}

\vspace{0.5cm}

{\small{\textsc{R\'esum\'e} :
Soit $A$ la jacobienne d'une courbe hyperelliptique d\'efinie sur un corps de nombres $k$. On donne une formule de d\'ecomposition de la hauteur de Faltings de $A$ et de la hauteur de N\'eron-Tate des points $k$-rationnels de $A$. On propose de plus en \ref{Bogomolov} une question de type Bogomolov sur l'espace de modules $\mathcal{A}_g$ des vari\'et\'es ab\'eliennes principalement polaris\'ees de dimension $g$.}}
\end{centering}

\vspace{0.3cm}

\begin{centering}
{\small{\textsc{Abstract} : Let $A$ be the jacobian variety of a hyperelliptic curve defined over a number field $k$. We provide a decomposition formula for the Faltings height of $A$ and for the N\'eron-Tate height of $k$-rational points on $A$. We formulate in \ref{Bogomolov} a question of Bogomolov type on the space $\mathcal{A}_g$ of principally polarized abelian varieties of dimension g.}}
\end{centering}

{\flushleft
}

\section{Introduction}

Soit $E$ une courbe elliptique d\'efinie sur $\overline{\mathbb{Q}}$ et donn\'ee dans un mod\`ele de Weierstrass. L'article \cite{CorSil} est d\'evolu \`a montrer une formule de d\'ecomposition en composantes locales de la hauteur de Faltings stable de $E$ en fonction des invariants classiques du mod\`ele de Weierstrass choisi, voir l'\'enonc\'e du th\'eor\`eme \ref{elliptique} ci-apr\`es.
\\

Le pr\'esent texte propose une g\'en\'eralisation en dimension sup\'erieure de cette formule. On traite des jacobiennes de courbes hyperelliptiques, un cadre o\`u la d\'efinition d'un discriminant est ais\'ee et o\`u ce discriminant joue le m\^eme r\^ole que dans le cas elliptique, reliant notamment des propri\'et\'es de bonne r\'eduction de la vari\'et\'e et des formules closes aux places archim\'ediennes en termes de fonctions th\^eta. A ce titre, les jacobiennes de courbes hyperelliptiques sont une g\'en\'eralisation naturelle. 
\\ 

Plus g\'en\'eralement, soit $A$ une vari\'et\'e ab\'elienne d\'efinie sur $\overline{\mathbb{Q}}$, principalement polaris\'ee, semi-stable de dimension $g\geq 1$, munie d'un fibr\'e $L$ ample et sym\'etrique. On s'int\'eresse \`a trois questions \'etroitement li\'ees. 
\\

Question 1 :  peut-on d\'ecomposer en composantes locales la hauteur diff\'erentielle de $A/k$ ? On s'accordera sur le fait qu'une composante locale en la place $v$ d'un corps $k$ de d\'efinition de $A$ est un nombre r\'eel calculable \`a partir des seules donn\'ees locales de la vari\'et\'e $A$ en $v$. 
\\

Question 2 : peut-on donner des formules explicites pour les hauteurs locales canoniques de N\'eron d'un point $\overline{\mathbb{Q}}$-rationnel sur $A$ ? L'existence de cette d\'ecomposition remonte \`a N\'eron \cite{Neron}.
\\

Question 3 : \'etant donn\'e un point $P\in{A(\overline{\mathbb{Q}})}\subset \mathbb{P}^N(\overline{\mathbb{Q}})$, comment estimer la diff\'erence $\hat{h}(P)-h(P)$ pour un choix de hauteur projective $h$ ? Citons par exemple les travaux \cite{CrePriSik} et \cite{Bru} en dimension 1, \cite{FlySma} et \cite{Stoll} en dimension 2. L'objectif \'etant d'obtenir en dimension quelconque de meilleures bornes que celles existantes dans \cite{MaZa}. 
\\

On cherche ici à apporter une r\'eponse possible \`a ces trois questions, r\'eponse provenant de la construction des mod\`eles de Moret-Bailly des vari\'et\'es ab\'eliennes. On trouvera le n\'ecessaire les concernant dans la section \ref{baseara} du pr\'esent texte. Ces mod\`eles de Moret-Bailly (ou MB-mod\`eles) jouaient d\'ej\`a un r\^ole important dans le travail \cite{Paz2, boda} o\`u leur d\'efinition est aussi rappel\'ee en d\'etails. 

L'article se veut accessible et comporte ainsi plusieurs paragraphes de rappels. Il est organis\'e comme suit. On pr\'esente les formules en d\'etails dans la section \ref{presentation}. La section \ref{base} d\'ecrit les \'enonc\'es de la th\'eorie des hauteurs de points alg\'ebriques utiles \`a la suite. La section \ref{baseara} pr\'esente une partie de la th\'eorie des mod\`eles de Moret-Bailly, permettant de calculer la hauteur d'un point alg\'ebrique par la formule clef dans la section \ref{formuleclef}. Dans la section \ref{hyper} on donne une formule explicite valable dans le cas des jacobiennes de courbes hyperelliptiques. Finalement la section \ref{dim1} montre que toutes les formules propos\'ees pour la hauteur d'une courbe elliptique fournissent le m\^eme r\'esultat.
\\

\section{Pr\'esentation des d\'ecompositions}\label{presentation}

Dans tout le texte on note $M_{k}$ l'ensemble des places du corps de nombres $k$ et $M_{k}^{\infty}$ l'ensemble de ses places archim\'ediennes, $M_{k}^{0}$ d\'esignant l'ensemble de ses places finies. On note $d=[K:\mathbb{Q}]$ son degr\'e. Pour toute place $v$ de $k$ on note $k_{v}$ le compl\'et\'e de $k$ pour la valuation $|.|_{v}$ associ\'ee o\`u on fixe $|p|_{v}=p^{-1}$ pour toute place finie $v$ au-dessus d'un nombre premier $p$. On note de plus $d_v=[k_v:\mathbb{Q}_v]$ pour le degr\'e local. Pour tout vecteur $x=(x_1, ..., x_n)$ de $k_v^n$ on pose  
\begin{center}
$\Vert x\Vert_v=$
$\left\{\begin{tabular}{ll}
$\displaystyle{\left(\sum_{i=1}^n\vert x_i\vert_v^2\right)^{\frac{1}{2}}}\quad \textrm{si $v$ est archim\'edienne,}$\\
$\displaystyle{\max_{1\leq i\leq n}\{\vert x_i\vert_v\}}\quad \mathrm{ sinon.}$
\end{tabular}\right.$
\end{center}

On travaillera avec la \emph{hauteur diff\'erentielle positive} ou \emph{hauteur de Faltings positive} d\'efinie par $\hFplus(A)=\frac{g}{2}\log(2\pi^2)+\hF(A)$, o\`u $\hF(A)$ est la hauteur introduite par Faltings dans \cite{Fal}. Cette version de la hauteur diff\'erentielle poss\`ede l'agr\'eable propri\'et\'e d'\^etre positive, voir \`a ce sujet la remarque \ref{norma} ci-dessous.

Le point de d\'epart de cette \'etude et l'origine de la question 1 est le cas de la dimension 1 o\`u on dispose de la formule suivante exprimant la hauteur diff\'erentielle positive d'une courbe elliptique, montr\'ee par Silverman dans l'ouvrage \cite{CorSil} page 254  (on a corrig\'e une puissance de $2\pi$ dans la d\'efinition du discriminant, voir par exemple la proposition 8.2 de \cite{deJo}, et chang\'e la normalisation des m\'etriques ici) :

\begin{thm}(Silverman)\label{elliptique}
Soit $E$ une courbe elliptique sur un corps de nombres $k$ de degr\'e $d$. Alors on a
\[
\hFplus(E)=\frac{1}{12d}\left[\log N_{k/\mathbb{Q}}(\Delta_E)-\sum_{v\in{M_k^{\infty}}}d_v \log\Big(\vert \Delta(\tau_v) \vert (2\Ima\tau_v)^6\Big)\right],
\]
o\`u $\Delta_E$ est le discriminant minimal de $E$ ; aux places archim\'ediennes $\tau_v$ est une matrice de p\'eriodes associ\'ee \`a $E(\overline{k}_v)$ et $\Delta(\tau_v)=q\prod_{n=1}^{+\infty}(1-q^n)^{24}$ est le discriminant modulaire, en ayant pos\'e $q=e^{2\pi i\tau_v}$.
\end{thm}



%

On cherche \`a g\'en\'eraliser cette formule en toute dimension. Dans le texte \cite{Aut2}, Autissier prouve une formule valable en dimension g\'en\'erale dans le cas o\`u la vari\'et\'e a potentiellement bonne r\'eduction partout:

\begin{thm}(Autissier)\label{Autissier}
Soit $A$ une vari\'et\'e ab\'elienne de dimension $g$ sur un corps de nombres $k$. Supposons que $A$ a potentiellement bonne r\'eduction partout. Soit $\Theta$ un diviseur sym\'etrique et ample sur $A$, d\'efinissant une polarisation principale $\lambda$. On pose $L=\mathcal{O}_A(\Theta)$ et on note $\mu$ la mesure de Haar de $A(\mathbb{C})$ de masse 1. Alors
$$\hFplus(A)=2g\hat{h}_L(\Theta)+\frac{2}{d}\sum_{v\in{M_k^{\infty}}}d_v I(A_v,\lambda_v),$$
o\`u $\displaystyle{I(A_v,\lambda_v)=-\int_{A(\overline{k}_v)}\log\Vert s\Vert_v \mu+\frac{1}{2}\log\int_{A(\overline{k}_v)}\Vert s\Vert_v^2 \mu}$ est positif et ind\'ependant du choix de section $s$ non nulle du  fibr\'e $L_v$.
\end{thm}

L'article \cite{Aut2} contient de plus un r\'esultat de d\'ecomposition de $h_{F^+}(A)-2g\hat{h}_L(\Theta)$ inconditionnel lorsque $A$ est un produit de courbes elliptiques et de surfaces ab\'eliennes. On voudrait arriver \`a une formule explicite tout en se passant de l'hypoth\`ese de bonne r\'eduction.

On traite partiellement la premi\`ere question dans la partie \ref{hyper} avec une d\'ecomposition explicite de la hauteur de Faltings pour les jacobiennes de courbes hyperelliptiques de genre $g$. Cette formule n'appara\^it pas dans la litt\'erature mais est probablement connue des experts. Comme en dimension 1, le calcul est facilit\'e par un choix de section tr\`es agr\'eable et bas\'e sur l'existence d'un discriminant de la courbe caract\'erisant la mauvaise r\'eduction aux places finies et d\'ecrit comme une forme modulaire aux places infinies. Il faut cependant tenir compte du comportement de cette section le long du bord de l'espace de modules de courbes. C'est la g\'en\'eralisation de la formule de Ueno de l'article \cite{Ueno}, \'etablie ici gr\^ace \`a l'utilisation des articles de Lockhart \cite{Lock}, Kausz \cite{Kausz} et de Jong \cite{Dej}. Quelques notations tout d'abord : pour $m\in{\frac{1}{2}\mathbb{Z}^{2g}}$ on pose $\varphi_{m}(\tau)=\theta_{m}(0,\tau)^{8},$ o\`u $\theta_{m}(z,\tau)$ est la fonction th\^eta de caract\'eristique $m$ associ\'ee au r\'eseau de
dimension $g$ dont la d\'efinition est rappel\'ee en (\ref{fonction theta}). Si $S$ est un sous-ensemble de $\{1,2,...,2g+1\}$ on d\'efinit alors
$\displaystyle{m_{S}=\sum_{i\in{S}}m_{i}\in{\frac{1}{2}\mathbb{Z}^{2g}}}$ avec :

\begin{center}
\begin{tabular}{llll}
$m_{2i-1}$ & $=$ & $\left[ \begin{array}{ccccccc}
^{t}(0 & ... & 0 & \frac{1}{2} & 0 & ... & 0) \\
^{t}(\frac{1}{2} & ... & \frac{1}{2} & 0 & 0 & ... & 0) \\ 	
\end{array}\right],$ & $1\leq i \leq g+1 \, ,$ \\
\\

\end{tabular}
\end{center}

\begin{center}
\begin{tabular}{llll}
$m_{2i}$ & $=$ & $\left[ \begin{array}{ccccccc}
^{t}(0 & ... & 0 & \frac{1}{2} & 0 & ... & 0) \\
^{t}(\frac{1}{2} & ... & \frac{1}{2} & \frac{1}{2} & 0 & ... & 0) \\ 	
\end{array}\right],$ & $1\leq i \leq g \, ,$ \\
\\

\end{tabular}
\end{center}

\noindent o\`u le coefficient non nul de la premi\`ere ligne est en i-\`eme
position. Soit alors $\mathcal{T}$ la collection des sous-ensembles de
$\{1,...,2g+1\}$ de cardinal $g+1$. Soit $U=\{1,3,...,2g+1\}$ et
notons $\circ$ l'op\'erateur de diff\'erence sym\'etrique. On d\'efinit alors
:

\begin{equation}\label{phitheta}
\varphi(\tau)=\prod_{T\in{\mathcal{T}}}\varphi_{m_{T\circ U}}(\tau).
\end{equation}

\begin{thm}\label{hauteur faltings jacobienne}
Soient $k$ un corps de nombres et $C$ une courbe hyperelliptique d\'efinie sur $k$, semi-stable et de genre $g\geq1$. On note $J_C$ sa jacobienne. On pose $l=\binom{2g+1}{g+1}$. Pour chaque place archim\'edienne $v$ on note $\tau_v$ une matrice de l'espace de Siegel\footnote{La d\'efinition de cet espace est rappel\'ee dans la section suivante.} ${\mathfrak S}_g$ telle que ${\mathcal A}_{v}(\mathbb{C})\simeq\mathbb{C}^g/(\mathbb{Z}^g+\tau_{v}\mathbb{Z}^g)$ 
comme vari\'et\'es ab\'eliennes principalement polaris\'ees et $\Delta_{\mathrm{min}}$ le discriminant minimal de $C/k$. Il existe des entiers $e_v\geq 0$ et des r\'eels $f_v\geq 0$ tels que $(8g+4)\cdot f_v=g \cdot \mathrm{ord}_v(\Delta_{\mathrm{min}})-(8g+4)e_v$ et tels que
la hauteur de Faltings de $J_C$ soit donn\'ee par la formule :
$$d\cdot\hFplus(J_C)=\sum_{v\vert\Delta_{\mathrm{min}}}d_v f_v\log\Nk(v)-\sum_{v\in{M_k^{\infty}}}d_v \log\left(2^{-\frac{2g}{8g+4}}|\varphi(\tau_{v})|^{\frac{1}{4l}}\det(\Ima
\tau_{v})^{\frac{1}{2}}\right),$$
o\`u $\varphi$ est un produit explicite de constantes th\^eta donn\'e par la formule (\ref{phitheta}). Dans le cas particulier des surfaces ab\'eliennes, on obtient
$$d\cdot\hFplus(J_C)=\sum_{v\vert\Delta_{\mathrm{min}}}d_v f_v\log\Nk(v)-\sum_{v\in{M_k^{\infty}}}d_v \log\left(2^{-\frac{1}{5}}|J_{10}(\tau_{v})|^{\frac{1}{10}}\det(\Ima
\tau_{v})^{\frac{1}{2}}\right),$$ avec $\displaystyle{J_{10}(\tau)=\prod_{m\in{Z_2}}\theta_m(0,\tau_v)^2}$, le produit portant sur l'ensemble $Z_2$ des 10 caract\'eristiques th\^eta paires en dimension 2.
\end{thm}

Pour une courbe elliptique, toutes ces d\'ecompositions co\"incident avec le th\'eor\`eme \ref{elliptique}, ce fait est v\'erifi\'e dans la section \ref{dim1}. Pour les surfaces ab\'eliennes, cela donne une version d\'emontr\'ee de la formule qu'on peut trouver dans l'article \cite{Ueno} page 765, en pr\'ecisant les contributions aux places finies et les conventions pour $J_{10}$. La comparaison des termes locaux entre la formule d'Autissier et les formules de Ueno et du th\'eor\`eme \ref{hauteur faltings jacobienne} est moins ais\'ee car le terme $\hat{h}_L(\Theta)$ n'est pas trivial en dimension $g\geq 2$.

On se donne \`a pr\'esent un MB-mod\`ele de niveau $r=4$ sur $k$, qui est un bon cadre pour obtenir l'\'enonc\'e suivant. Un tel mod\`ele existe toujours modulo une \'eventuelle extension finie de corps de nombres. On suppose en particulier que les points de $16$-torsion de $A$, dont l'ensemble est not\'e $A[16]$, sont rationnels sur $k$.  On pourra consulter la section \ref{baseara} (voir aussi l'article \cite{Paz2} et ses r\'ef\'erences) pour une pr\'esentation d\'etaill\'ee de ces mod\`eles et de leurs propri\'et\'es. On montre alors comme corollaire direct du th\'eor\`eme \ref{hauteur faltings jacobienne} et de la proposition \ref{clef} :

\begin{cor}\label{coro1}
Soit $A$ une jacobienne de courbe hyperelliptique de dimension $g$ d\'efinie sur $\overline{\mathbb{Q}}$, semi-stable et munie d'un fibr\'e $L$ ample et sym\'etrique portant une polarisation principale. Soit $R\in{A(\overline{\mathbb{Q}})}$, soit $k$ un corps de d\'efinition de $A$, $A[16]$, $R$. Alors on a la formule
\[
\widehat{h}_{A,L^{\otimes 16}}(R)  -  h_{A,L^{\otimes 16}}(R)=\frac{1}{d}\sum_{v\in{M_k}} d_v \alpha_{A,L,v}(R)
\]
o\`u $\alpha_{A,L,v}(R)$ est donn\'e par la diff\'erence des termes locaux dans la d\'efinition \ref{beta3} et dans le th\'eor\`eme \ref{hauteur faltings jacobienne}, en tenant compte du facteur $1/2$.
\end{cor}

Ce corollaire est donc une r\'eponse possible pour les questions 2 et 3 propos\'ees plus haut. Obtenir de telles d\'ecompositions permet de mener des calculs explicites place par place.
Cela a des applications, par exemple dans le proc\'ed\'e de saturation du groupe de Mordell-Weil (\textit{i.e.} la recherche de g\'en\'erateurs explicites du groupe des points rationnels), o\`u il est important de savoir estimer la diff\'erence entre hauteur canonique et hauteur na\"ive d'un point rationnel. On sait depuis Manin-Zarhin \cite{MaZa} pour les points et David-Philippon \cite{DavPhi} pour les sous-vari\'et\'es que la valeur absolue de cette diff\'erence est \textit{major\'ee} par la hauteur de la vari\'et\'e ab\'elienne ambiante, \`a une constante explicite pr\`es. On trouve dans la proposition \ref{clef} (ou corollaire \ref{coro1} ci-dessus) une \textit{\'egalit\'e} permettant d'affirmer, avec le th\'eor\`eme \ref{hauteur faltings jacobienne}, qu'on peut estimer cette diff\'erence en menant des calculs locaux.

\subsection*{Remerciements}
L'auteur remercie chaleureusement Pascal Autissier et Ga\"el R\'emond pour leurs conseils pr\'ecieux. Merci \`a l'arbitre de publication pour son travail. L'auteur est soutenu par le programme ANR-14-CE25-0015 Gardio, par ANR-17-CE40-0012 Flair et par la chaire Niels Bohr DNRF de Lars Hesselholt.

\section{Hauteurs globales, hauteurs locales}\label{base}

On rappelle dans cette section un th\'eor\`eme de N\'eron, la d\'efinition de l'espace de Siegel et la d\'efinition de la hauteur de Faltings.

\subsection*{Th\'eor\`eme de N\'eron}
Une hauteur de Weil $h_{A,\mathcal{D}}$ associ\'ee \`a un diviseur $\mathcal{D}$ sur une vari\'et\'e ab\'elienne $A/k$ est par d\'efinition une somme index\'ee par les places de $k$ de fonctions $\lambda_{\mathcal{D},v}$ \`a valeurs r\'eelles (d\'efinies hors du diviseur $\mathcal{D}$). C'est une fonction v\'erifiant la relation suivante (issue du th\'eor\`eme du cube) : il existe une constante $c$ telle que pour tous points $P,Q,R \in{A(k)}$, et en notant temporairement $h=h_{A,\mathcal{D}}$ :
$$\Big|h(P+Q+R)-h(P+Q)-h(Q+R)-h(R+P)+h(P)+h(Q)+h(R)\Big| \leq c.$$
Si on suppose de plus que le diviseur $\mathcal{D}$ est sym\'etrique on obtient (en prenant $R=-Q$) une relation de quasi-parall\'elogramme :
$$\Big|h(P+Q)+h(P-Q)-2h(P)-2h(Q)\Big| \leq c.$$
Le passage \`a la limite effectu\'e pour d\'efinir la hauteur de N\'eron-Tate permet d'obtenir $c=0$. Cette construction offre donc l'avantage suivant : la hauteur de N\'eron-Tate devient une forme quadratique, dont le cône isotrope est le sous-groupe de torsion de la vari\'et\'e ab\'elienne. Le th\'eor\`eme suivant de N\'eron offre la possibilit\'e de d\'ecomposer cette hauteur canonique aussi (voir \cite{HiSi} page 242) :

\begin{thm} \label{decomposition}(N\'eron)  Soit $A/k$ une vari\'et\'e ab\'elienne d\'efinie sur un corps de
nombres $k$. Soit $M_{k}$ l'ensemble des places de $k$. Pour tout diviseur $\mathcal{D}$ sur $A$ on note $A_{\mathcal{D}}=A\backslash\supp(\mathcal{D})$. Alors pour toute place
$v\in{M_{k}}$ il existe une fonction hauteur locale, unique \`a une fonction constante pr\`es :
$$\widehat{\lambda}_{\mathcal{D},v}:A_{\mathcal{D}}(k_{v})\longrightarrow \mathbb{R}, $$
appel\'ee hauteur locale canonique, d\'ependant du choix de $\mathcal{D}$ et
v\'erifiant les propri\'et\'es suivantes, avec $\gamma_{i,v}$ des constantes d\'ependant de $v$ :

\begin{enumerate}
\item $ \widehat{\lambda}_{\mathcal{D},v}-\lambda_{\mathcal{D},v}$ est une fonction born\'ee.
\item Pour tous diviseurs $\mathcal{D}_{1}$ et $\mathcal{D}_{2}$ sur $A$, on a $ \widehat{\lambda}_{\mathcal{D}_{1}+\mathcal{D}_{2},v}= \widehat{\lambda}_{\mathcal{D}_{1},v}+
  \widehat{\lambda}_{\mathcal{D}_{2},v}+\gamma_{1,v} $.
\item Si $\mathcal{D}=\divi(f)$, alors $\widehat{\lambda}_{\mathcal{D},v}=v\circ f+\gamma_{2,v}$, o\`u $v(.)=-\log\vert .\vert_v$.
\item Si $\Phi:B\rightarrow A$ est un morphisme entre deux vari\'et\'es ab\'eliennes
  alors on a la relation  $ \widehat{\lambda}_{\Phi^{*}\mathcal{D},v}=
  \widehat{\lambda}_{\mathcal{D},v}\circ \Phi+\gamma_{3,v}$.
\item Soit $Q \in{A(k)}$ et soit $t_{Q}: A \rightarrow A $ la translation
  par $Q$. Alors on a la relation : $\widehat{\lambda}_{t_{Q}^{*}\mathcal{D},v}=
  \widehat{\lambda}_{\mathcal{D},v}\circ t_{Q}+\gamma_{4,v}$.
\item Soit $\widehat{h}_{A,\mathcal{D}}$ la hauteur globale canonique de $A$ associ\'ee
  \`a $\mathcal{D}$. Il existe une constante $\hat{c}$ telle que, pour tout $P\in{A_{\mathcal{D}}(k)}$ :
  $$ \widehat{h}_{A,\mathcal{D}}(P)=\frac{1}{d}\sum_{v\in{M_{k}}}d_v \widehat{\lambda}_{\mathcal{D},v}(P) + \hat{c}. $$
\item{Si $\mathcal{D}$ v\'erifie $[2]^{*}\mathcal{D}=4\mathcal{D}+\divi(f)$ pour $f$ une fonction rationnelle sur $A$ et si l'on fixe les constantes de telle sorte que, pour tout $P\in{A_\mathcal{D}}$ avec $[2]P\in{A_{\mathcal{D}}}$, on ait la relation $(*)\;\widehat{\lambda}_{\mathcal{D},v}([2]P)=4\widehat{\lambda}_{\mathcal{D},v}(P)+v(f(P))$, alors pour tout $P\in{A_{\mathcal{D}}}$:
$$\widehat{h}_{A,\mathcal{D}}(P)=\frac{1}{d}\sum_{v\in{M_{k}}}d_v \widehat{\lambda}_{\mathcal{D},v}(P).$$
(Notons que $f$ est unique \`a multiplication par une constante $a\in{k^{*}}$ pr\`es. Notons aussi que la relation $(*)$ permet de fixer la constante $\hat{c}=0$ dans l'item pr\'ec\'edent.) }
 
\end{enumerate} 

\end{thm}

\subsection*{Espace de Siegel et fonctions th\^eta}\label{domaine de Siegel}

Soit $v$ une place archim\'edienne. On notera ${\mathfrak S}_{g}$ l'espace de Siegel associ\'e aux vari\'et\'es
ab\'eliennes sur $\overline{k}_{v}$ principalement polaris\'ees de dimension $g$ et munies d'une base symplectique (on pourra consulter \cite{BiLa} page 213). C'est l'ensemble des matrices $\tau=\tau_{v}$ de taille $g\times g$ sym\'etriques \`a coefficients complexes et v\'erifiant la condition $\Ima \tau >0$ (\textit{i.e.} d\'efinies positives). Cet espace est muni d'une action transitive du groupe symplectique $\Gamma = \Sp(2g,\mathbb{R})$ donn\'ee par :
$$\left[ \begin{array}{cc}
A & B   \\
C & D   \\ 	
\end{array}\right]\!\cdot\!\tau=(A\tau+B)(C\tau+D)^{-1}.$$

On consid\`ere alors $F_{g}$ un domaine fondamental pour l'action du sous-groupe $\Sp_{2g}(\mathbb{Z})$. On peut choisir $F_{g}$ de telle sorte qu'une matrice $\tau$ de ce domaine v\'erifie en particulier les conditions suivantes (voir \cite{Frei} page 34) :

\begin{itemize}

\item[$\bullet$]{S1 : Pour tout $\sigma\in{\Sp_{2g}(\mathbb{Z})}$ on a : $\det(\Ima\sigma \cdot \tau)\leq \det(\Ima\tau)$. On dira que $\Ima\tau$ est \emph{maximale} pour l'action de $\Sp_{2g}(\mathbb{Z})$.}

\item[$\bullet$]{S2 : Si $\Ree\tau=(a_{i,j})_{1\leq i, j \leq g}$ alors $|a_{i,j}|\leq \frac{1}{2}$ pour tous $1\leq i,j\leq g$.}

\item[$\bullet$]{S3 : Si $\Ima\tau=(b_{i,j})$ alors pour tout $l\in\{1,...,g\}$ et tout $\zeta=(\zeta_{1},...,\zeta_{g})\in{\mathbb{Z}}^{g}$ tel que $\pgcd(\zeta_{l},...,\zeta_{g})=1$ on a $^{t}\zeta (\Ima\tau) \zeta \geq b_{l,l}$. De plus pour tout $i\in\{1,...,g\}$ on a $b_{i,i+1}\geq 0$.}

\end{itemize}

Ces conditions impliquent $b_{g,g}\geq...\geq b_{1,1}\geq \sqrt{3}/2$ et $ b_{i,i}/2\geq|b_{i,j}|$.
On d\'efinit alors pour $z\in{\mathbb{C}^g}$ et $\tau\in{{\mathfrak S}_g}$ les s\'eries th\^eta avec caract\'eristiques $a,b\in{\mathbb{R}^g}$ :

\begin{equation}\label{fonction theta}
\displaystyle{\theta_{a,b}(z,\tau)=\sum_{n\in{\mathbb{Z}^g}}e^{i\pi\,(n+a)'\tau(n+a)+2i\pi\,(n+a)'(z+b)},}
\end{equation}
o\`u $x'$ est le vecteur ligne transpos\'e du vecteur colonne $x\in{\mathbb{R}^g}$.
\noindent On dira que les matrices de $F_g$ sont r\'eduites au sens de Siegel.

\subsection*{Hauteur de Faltings positive}

Soient $k$ un corps de nombres de degr\'e $d$ et $S=\Spec(\mathcal{O}_{k})$ le spectre de son anneau d'entiers. Un \emph{fibr\'e vectoriel m\'etris\'e} de rang $r$ sur $S$ est un $\mathcal{O}_{k}$-module projectif $\mathcal{L}$ de rang $r$ muni d'une collection $\{||.||_{v}\}_{v\in{M_{k}^{\infty}}}$ telle que $||.||_{v}$ soit une norme hermitienne sur le $k_{v}$-espace vectoriel $\mathcal{L}\otimes_{\mathcal{O}_{k}}\overline{k}_{v}$, v\'erifiant $||x||_{v}=||\overline{x}||_{\overline{v}}$ pour tout plongement $v : k \hookrightarrow \mathbb{C}$.

Le \emph{degr\'e d'Arakelov} d'un fibr\'e en droites m\'etris\'e $(\mathcal{L},||.||_{v})$ sur $S$ est d\'efini, en prenant un \'el\'ement non nul $s\in{\mathcal{L}}$ :
$$\degrar(\mathcal{L})=\log\Card\Big(\mathcal{L}/s \mathcal{O}_{k}\Big)-\sum_{v \in{M_k^{\infty}}}d_v \log \Vert s \Vert_{v}.$$
La formule du produit nous assure que ce degr\'e ne d\'epend pas globalement du choix de section $s$ non nulle (mais les formules locales d\'ependent de la section bien entendu).

Soit alors $A/k$ une vari\'et\'e ab\'elienne de dimension $g\geq 1$. Soient $\mathcal{A}\rightarrow S$ son mod\`ele de N\'eron, $\varepsilon : S\rightarrow \mathcal{A}$ sa section neutre et $\Omega_{\mathcal{A}/S}^{g}$ le faisceau des $g$-formes diff\'erentielles, qui est localement libre de rang $1$. On pose $\omega_{\mathcal{A}/S}=\varepsilon^{*}(\Omega_{\mathcal{A}/S}^{g})$ ; c'est un fibr\'e en droites sur $S=\Spec(\mathcal{O}_{k})$ qu'on peut identifier au module de ses sections globales. On munit ce fibr\'e des m\'etriques suivantes :
\begin{equation}\label{metrique}
\forall \alpha\in{\omega_{\mathcal{A}/S}\otimes_{v}\mathbb{C}},\;\; ||\alpha||_{v}^{2}=\frac{i^{g^2}}{(2\pi)^{2g}}\int_{A_{v}(\mathbb{C})}\alpha \wedge \overline{\alpha},
\end{equation}
\noindent o\`u on a identifi\'e $\alpha$ \`a une section globale de $\Omega_{\mathcal{A}/S}^{g}$. On notera qu'on a choisi ici d'\'elever $2\pi$ \`a la puissance $2g$, voir la discussion plus bas. On d\'efinit alors :

\begin{defin} \label{hauteur de Faltings relative}
Soit $A/k$ une vari\'et\'e ab\'elienne semi-stable d\'efinie sur un corps de nombres $k$. On appelle \emph{hauteur de Faltings positive} la quantit\'e :
$$\hFplus(A)=\frac{1}{d}\,\degrar(\omega_{\mathcal{A}/S}).$$
\end{defin}
\noindent
On a donc la relation $$\hFplus(A)=\frac{g}{2}\log(2\pi^2)+\hF(A),$$ o\`u $\hF(A)$ est la hauteur de Faltings de \cite{Fal}.

\begin{rem}\label{norma}
\emph{Un point sur la normalisation des m\'etriques hermitiennes.} 

La puissance $i^{g^2}$ au num\'erateur est en fait $i^g(-1)^{\frac{g(g-1)}{2}}$, le terme $-1$ provient du caract\`ere altern\'e du produit ext\'erieur. Il y a ensuite (au moins) cinq possibilit\'es int\'eressantes pour le d\'enominateur.
\begin{itemize}
\item[A.\,] $2^g$ : cela permet de simplifier le passage des coordonn\'ees complexes aux coordonn\'ees r\'eelles dans les calculs de volumes, voir par exemple l'article de Chai dans \cite{CorSil} page 250. C'est le choix de \cite{Fal}.
\item[B.\,] $(2\pi)^g$ : permet d'obtenir une formule de hauteur dans le cas CM qui ne fait pas intervenir de puissance de $\pi$, d'apr\`es la formule de Chowla-Selberg. Voir par exemple \cite{Del} page 29.
\item[C.\,] $(2\pi)^{2g}$ : assure que la hauteur de Faltings est toujours positive, d'apr\`es une in\'egalit\'e de Bost, voir \cite{Bost} ou \cite{Aut2}, ainsi que des d\'etails de preuve dans \cite{GauR}.
\item[D.\,] $(4\pi)^{\frac{2g}{3}}$ : supprime le terme constant dans la formule de Noether donn\'ee dans \cite{MB}, si on garde la normalisation traditionnelle du $\delta$ de Faltings.
\item[E.\,] $1$ : \'evite les constantes dans la d\'efinition initiale. C'est le choix fait dans \cite{Col}.
\end{itemize}
\vspace{0.2cm}

\noindent Dans ce travail, nous avons opt\'e pour le choix C. Un choix qui conduit plus naturellement \`a la question : les vari\'et\'es de hauteur minimale ont-elle une structure particuli\`ere ? Ce que l'on pourrait formuler de la mani\`ere suivante :

\end{rem}

\begin{ques}(Bogomolov sur $\mathcal{A}_g$)\label{Bogomolov}
Peut-on trouver un r\'eel explicite \textbf{optimal} $\varepsilon_g>0$ qui ne d\'epende que de $g$ et tel que pour toute vari\'et\'e ab\'elienne $A$ sur $\overline{\mathbb{Q}}$ de dimension $g$ on ait $h_{F^{+}}(A)\geq \varepsilon_g$ ? Pour quelle d\'ependance en $g$ ? Est-ce que $\varepsilon_g$ est toujours atteint pour une vari\'et\'e admettant des multiplications complexes ?
\end{ques}

\noindent Nous savons d'ores et d\'ej\`a que l'additivit\'e de la hauteur va fournir des contraintes sur la famille $(\varepsilon_g)_{g\geq 1}$. Par exemple si $E$ est une courbe elliptique sur $\overline{\mathbb{Q}}$ on aura $h_{F^{+}}(E^g)\geq \varepsilon_g$, ce qui implique $g\varepsilon_1\geq \varepsilon_g$. Une formule du type $\varepsilon_g=c_0 g$, avec $c_0$ une constante universelle, serait fonctorielle. La litt\'erature nous renseigne dans la remarque suivante sur l'existence de ``petits points''.

\begin{rem}
On donne ici des exemples de calculs de la hauteur de Faltings avec le choix de m\'etrique fait ici :
\begin{enumerate}
\item (D'apr\`es Bost, Mestre, Moret-Bailly \cite{BoMeMo} page 93) On observe la courbe $C$ de genre 2 donn\'ee par l'\'equation affine $y^2+y=x^5$ sur un corps de nombres sur laquelle elle est semi-stable, alors $$\hFplus(J_C)=3\log2\pi-\frac{1}{2}\log\left(\Gamma\Big(\frac{1}{5}\Big)^{5}\Gamma\Big(\frac{2}{5}\Big)^{3}\Gamma\Big(\frac{3}{5}\Big)\Gamma\Big(\frac{4}{5}\Big)^{-1}\right),$$ donc de valeur approch\'ee $\hFplus(J_C)=0,38537...\geq \varepsilon_2.$
\item (D'apr\`es Chowla-Selberg, voir Deligne \cite{Del} page 29) Si $E$ est une courbe elliptique semi-stable sur un corps de nombres et \`a multiplication complexe par l'anneau des entiers de $\mathbb{Q}(\sqrt{-D})$ o\`u $-D$ est le discriminant, on note $\epsilon$ le caract\`ere quadratique de Dirichlet, $w$ le nombre d'unit\'es et $h$ le nombre de classes, alors $$\hFplus(E)=\frac{1}{2}\log{2\pi}-\frac{1}{2}\log\left(\frac{1}{\sqrt{D}}\left[\prod_{0<a<D}\Gamma\Big(\frac{a}{D}\Big)^{\epsilon(a)}\right]^{\frac{w}{2h}} \right),$$ en particulier si on choisit $D=3$ on trouve la valeur approch\'ee $\hFplus(E)=0,16993...\geq \varepsilon_1.$
\end{enumerate}
\end{rem}

\section{Mod\`eles de Moret-Bailly des vari\'et\'es ab\'eliennes}
\label{baseara}
On rappelle dans ce paragraphe une partie de la th\'eorie des mod\`eles de Moret-Bailly, voir \cite{Moret} et
\cite{Moret2}. Cela reprend des travaux de \cite{Bost2}, \S~4.~2, dans la m\^eme veine que \cite{Paz2, boda}.

\subsection{D\'efinitions}
\label{defini}
Soit $k$ un corps de nombres, ${\mathcal O}_k$ son anneau d'entiers et
$\pi\colon{\mathcal A}\longrightarrow \Spec({\mathcal O}_k)$ un sch\'ema en groupes semi-stable,
{\it i. e.} un sch\'ema en groupes lisse de type fini et s\'epar\'e sur $\Spec({\mathcal
O}_k)$,  tel que les composantes de ses fibres sont des extensions de vari\'et\'es ab\'eliennes par des tores.


Soit $\mathcal{L}$ un fibr\'e en droite sur $\mathcal{A}$. Notons
$$p_i\colon {\mathcal L}^3:={\mathcal L}\times_{{\mathcal O}_k}{\mathcal L}\times_{{\mathcal O}_K}{\mathcal L}\longrightarrow {\mathcal L}\, , \hspace{.5cm} i=1,2,3 $$
les projections sur les trois facteurs, notons
$$ p_I\colon {\mathcal A}^3\longrightarrow{\mathcal A}$$
le morphisme envoyant un point g\'eom\'etrique $(x_1,x_2,x_3)$ sur $\sum_{i\in I}x_i$, pour tout sous-ensemble non vide $I$ de $\{1,2,3\}$, et notons par ${\overline{{\mathcal O}_{{\mathcal A}^3}}}$ le fibr\'e hermitien trivial
 $({\mathcal O}_{{\mathcal A}^3}, \Vert.\Vert)$ donn\'e par $\Vert 1\Vert=1$. Alors,
un fibr\'e en droites hermitien $\overline{\mathcal L}$ sur $\mathcal A$ est cubiste si et seulement s'il existe un isomorphisme isom\'etrique
\begin{equation}
\label{cubist}
{\mathcal D}_3(\overline{\mathcal L}):=\bigotimes_{I\subset\{1,2,3\},I\neq\emptyset}
\left(p_{I}^{\star}{\overline{\mathcal
L}}\right)^{\otimes(-1)^{\#I}}\stackrel{\sim}{\longrightarrow}\overline{{\mathcal O}_{{\mathcal A}^3}}
\end{equation}
de fibr\'es en droites hermitiens sur ${\mathcal A}^3$ satisfaisant de bonnes conditions de sym\'etrie et de cocycle 
({\it confer} \cite{Moret}, I.~2.~4.~5., (i) et (iii)). La relation~(\ref{cubist}) implique que si 
${\varepsilon}\colon \Spec({\mathcal O}_k)\longrightarrow{\mathcal A}$ est la section neutre,
$$\varepsilon^{\star}\overline{\mathcal L}\simeq \overline{{\mathcal O}}_{\Spec({\mathcal O}_k)}\;,$$
et si ${\mathcal A}_k$ est une vari\'et\'e ab\'elienne, la $(1,1)$ forme $c_1(\overline{\mathcal L})$ est invariante par translation sur tous les tores complexes ${\mathcal A}_{\sigma}(\mathbb{C})$, pour
$\sigma\colon k\hookrightarrow \mathbb{C}$.

R\'eciproquement, si $\mathcal A$ est un sch\'ema ab\'elien sur
${\mathcal O}_k$, ces deux derni\`eres propri\'et\'es caract\'erisent les fibr\'es en droites hermitiens sur $\mathcal A$.

Soit $\pi\colon{\mathcal A}\longrightarrow \Spec({\mathcal O}_k)$ un sch\'ema en groupes semi-stable dont la fibre g\'en\'erique ${\mathcal A}_k$ est une vari\'et\'e ab\'elienne. Pour tout fibr\'e en droites $\mathcal M$ sur $\mathcal A$, l'image directe $\pi_{\star}{\mathcal M}$ est coh\'erente (voir \cite{Moret}, lemme~VI, I.~4.~2) et sans torsion,
donc localement libre. Si $\overline{{\mathcal L}_k}$ est un fibr\'e en droites cubiste hermitien sur $\mathcal A$ et si
${\mathcal L}_k$ est ample sur ${\mathcal A}_k$, alors $\mathcal L$ est ample on $\mathcal A$ (voir \cite{Ray},
th\'eor\`eme VIII.~2, et \cite{Moret}, proposition~VI.~2.~1) et $c_1({\mathcal L})$ est strictement positif sur
${\mathcal A}(\mathbb{C})$. En effet, elle est invariante par translation sur chaque composante de ${\mathcal A}(\mathbb{C})$  et cohomologue \`a une forme $(1,1)$ strictement positive. On peut ainsi d\'efinir $\pi_{\star}(\overline{\mathcal
L})$ comme le fibr\'e vectoriel hermitien de rang 
$$\rho({\mathcal L}_{\overline{\mathbb{Q}}}):=\frac{1}{g!}c_1({\mathcal L}_{\overline{\mathbb{Q}}})^g$$
sur $\Spec({\mathcal O}_k)$ tel que $\pi_{\star}({\mathcal L})$ muni de la structure hermitienne h\'erit\'ee de la m\'etrique $L^2$ not\'ee $\Vert.\Vert$ sur $\overline{\mathcal L}$ et la mesure de {Haar} normalis\'ee sur le tore complexe ${\mathcal A}_{\sigma}(\mathbb{C})$.  En d'autres termes, pour toute section $s\in\pi_{\star}{\mathcal L}\otimes_{\sigma}\mathbb{C}\simeq {\rm H}^2({\mathcal A}_{\sigma},{\mathcal
L}_{\sigma})$, on pose
$$\Vert s\Vert_{\sigma}^2=\int_{{\mathcal A}_{\sigma}(\mathbb{C})}\Vert s(x)\Vert^2_{\overline{\mathcal L}}d\mu(x)\;,$$
o\`u $d\mu$ d\'esigne la mesure de {Haar} normalis\'ee sur ${\mathcal A}_{\sigma}(\mathbb{C})$.

\begin{defin}\label{$MB$--model}
Soit $A$ une vari\'et\'e ab\'elienne sur $\overline{\mathbb{Q}}$, soit $L$ un fibr\'e en droites ample et sym\'etrique sur $A$ et
$F$ un sous-ensemble fini de $A(\overline{\mathbb{Q}})$. On appelle {\em $MB$-mod\`ele} de $(A,L,F)$ sur un corps de nombres $k\subset\overline{\mathbb{Q}}$ la donn\'ee suivante :
\begin{itemize}
\item[$\bullet$] un sch\'ema en groupes semi-stable $\pi\colon {\mathcal A}\longrightarrow \Spec({\mathcal O}_k)$,
\item[$\bullet$] un isomorphisme $i\colon { A}\rightarrow^{\!\!\!\!\!\!\sim} {\mathcal A}_{\overline{\mathbb{Q}}}$ de vari\'et\'es abeliennes
sur $\overline{\mathbb{Q}}$,
\item[$\bullet$] un fibr\'e hermitien cubiste ${\overline{\mathcal L}}$ sur $\mathcal A$,
\item[$\bullet$] un isomorphisme $\varphi$ comme pr\'ecis\'e en \cite{Paz2} dans la proposition 2.5.
\item[$\bullet$] pour tout $P\in F$, une section $\varepsilon_P\colon \Spec({\mathcal O}_k)\longrightarrow
{\mathcal A}$ de l'application $\pi$ telle que le point g\'eom\'etrique $\varepsilon_{P,\overline{\mathbb{Q}}}\in {\mathcal A}(\overline{\mathbb{Q}})$ \'egale le point $i(P)$,
\end{itemize}
satisfaisant de plus les conditions : il existe un sous-sch\'ema $\mathcal K$ de $\mathcal A$, plat et fini sur $\Spec({\mathcal O}_k)$, tel que $i^{-1}({\mathcal K}_{\overline{\mathbb{Q}}})$ \'egale le groupe de Mumford $K(L^{\otimes 2})$, \textit{i.e.} le sous-groupe alg\'ebrique fini de $A$ dont les points rationels $x$ sur $\overline{\mathbb{Q}}$ sont caract\'eris\'es par l'existence d'un isomorphisme de fibr\'es en droites sur $A$
$$t_x^{\star}L^{\otimes2}\simeq L^{\otimes2}\;.$$
\end{defin}

\begin{defin}\label{bon}
Pour tout triplet $(A,L,r)$ avec $A$ une vari\'et\'e ab\'elienne sur $\overline{\mathbb{Q}}$ et $L$ un fibr\'e en droites sym\'etrique et ample, $r>0$ un entier pair, on dit qu'un corps de nombres $k$ est $MB$ s'il existe un $MB$--mod\`ele du type $(\pi\colon{\mathcal A}\longrightarrow \Spec({\mathcal O}_k), i, {\overline{\mathcal L}},
\varphi,(\varepsilon_P)_{P\in A_{r^2}})$ rationnel sur $k$.
\end{defin}

\begin{rem}
On peut construire des corps de nombres $MB$ en utilisant le th\'eor\`eme de r\'eduction semi-stable (\textit{confer} \cite{Moret} Theorem 3.5 page 58).
\end{rem}

\subsection{Propri\'et\'es des $MB$--mod\`eles}
\label{propmb}
Les propri\'et\'es principales des $MB$--mod\`eles sont montr\'ees par
 Moret-Bailly en \cite{Moret} et \cite{Moret2}. Voir de plus Breen \cite{Breen} et Mumford \cite{Mum1}. En voici une pr\'esentation concise :

\begin{thm}
\label{prmbmod} Soit $A$ une vari\'et\'e ab\'elienne de dimension $g$ sur $\overline{\mathbb{Q}}$, soit $L$ un fibr\'e en droites sym\'etrique et ample sur $A$ et $F$ un sous-ensemble fini de $A({\overline{\mathbb{Q}}})$. On a

\begin{itemize}
\item[{\rm (i)}] {\em Existence}. Pour tout corps de nombres $k_0$, il existe un corps de nombres $k$ contenant $k_{0}$ et un $MB$--mod\`ele
$(\pi\colon{\mathcal A}\longrightarrow \Spec({\mathcal O}_k), i, {\overline{\mathcal L}},
\varphi,(\varepsilon_P)_{P\in F})$ pour les donn\'ees $(A,L,F)$.
\item[{\rm (ii)}] {\em{Hauteur de N\'eron-Tate}}. Pour tout $MB$--mod\`ele comme en {\rm (i)} et pour tout point $P\in F$, la hauteur normalis\'ee $[k:\mathbb{Q}]^{-1}\widehat{\degr}(\varepsilon^{\star}_P\overline{\mathcal L})$ \'egale la valeur en $P$ de la hauteur de N\'eron-Tate associ\'ee \`a $L$ et not\'ee $\widehat{h}_{L}(P)$.
\item[{\rm (iii)}] {\em Ind\'ependance des $MB$--mod\`eles}. Pour toute paire de $MB$--mod\`eles
$$(\pi\colon{\mathcal A}\longrightarrow \Spec({\mathcal O}_k), i, {\overline{\mathcal L}},
\varphi,(\varepsilon_P)_{P\in F})$$ et $$(\pi'\colon{\mathcal A}\longrightarrow
\Spec({\mathcal O}_k), i', {\overline{{\mathcal L}'}},
\varphi',(\varepsilon'_P)_{P\in F})$$ de $(A,L,F)$ sur un corps de nombres $k$, les isomorphismes canoniques 
$i$, $\varphi$, $i'$ et $\varphi'$~:
$$(\pi_{\star}{\mathcal L})_{\overline{\mathbb{Q}}}\simeq {\rm H}^0(A,L)\simeq (\pi'_{\star}{{\mathcal L})'}_{\overline{\mathbb{Q}}}$$
et
$$(\varepsilon_P^{\star}{\mathcal L})_{\overline{\mathbb{Q}}}\simeq  L_{\vert P}\simeq ({\varepsilon'_P}^{\star}{{\mathcal
L}'})_{\overline{\mathbb{Q}}}\hspace{1cm} (\forall P\in F)$$ 
s'\'etendent en isomorphismes isom\'etriques de fibr\'es en droites hermitiens sur la base $\Spec({\mathcal O}_k)$~:
$$\pi_{\star}({\overline{\mathcal L}})\simeq \pi'_{\star}({\overline{{\mathcal L}'}})$$
et
$${\varepsilon_P}^{\star}({\overline{\mathcal L}})\simeq {\varepsilon'_P}^{\star}({\overline{{\mathcal L}'}})\;.$$
\item[{\rm (iv)}] {\em Extensions des scalaires}. Soit
 $$(\pi\colon{\mathcal A}\longrightarrow \Spec({\mathcal O}_k), i, {\overline{\mathcal L}},
\varphi, (\varepsilon_P)_{P\in F})$$
 un $MB$-mod\`ele sur un corps de nombres $k$ et soit $k'$ un corps de nombres tel que $k\subset k' \subset\overline{\mathbb{Q}}$. On obtient par extension des scalaires de ${\mathcal O}_k$ \`a ${\mathcal O}_{k'}$ un sch\'ema en groupes semi-stable
$$\tilde{\pi}\colon\tilde{{\mathcal A}}:={\mathcal A}\times_{{\mathcal O}_k}{\mathcal
O}_{k'}\longrightarrow \Spec({\mathcal O}_{k'})\;,$$
un fibr\'e en droites hermitien $\overline{\tilde{\mathcal L}}$ sur $\tilde{\mathcal A}$ (en tirant en arri\`ere $\overline{\mathcal L}$
par la premi\`ere projection ${\mathcal A}\times_{{\mathcal O}_k}{\mathcal O}_{k'}\longrightarrow {\mathcal A}$), des sections
$$\tilde{\varepsilon}_P:=\varepsilon_P\otimes_{{\mathcal O}_k}{\mathcal
O}_{k'}\colon\Spec({\mathcal O}_{k'})\longrightarrow\tilde{\mathcal A}\;,$$
et les isomorphismes $i$ et $\varphi$ d\'eterminent des isomorphismes :
$$\tilde{i}\colon A\rightarrow^{\!\!\!\!\!\!\sim}{\tilde{\mathcal A}}_{\overline{\mathbb{Q}}} \hspace{.5cm} \mbox{\it et}\hspace{.5cm}
\tilde{\varphi}\colon L\rightarrow^{\!\!\!\!\!\!\sim}\tilde{i^{\star}}{\mathcal L}_{\overline{\mathbb{Q}}}\;.$$

Le 5-uplet $(\tilde{\pi}\colon \tilde{\mathcal
A}\longrightarrow \Spec({\mathcal O}_{k'}), \tilde{i}, {\overline{\tilde{\mathcal L}}},
\tilde{\varphi}, (\tilde{\varepsilon}_P)_{P\in F})$ est un $MB$-mod\`ele de $(A,L,F)$ sur $k'$.
De plus, si $j\colon \Spec({\mathcal O}_{k'})\longrightarrow\Spec({\mathcal O}_k)$ est l'application inclusion ${\mathcal O}_k\hookrightarrow{\mathcal O}_{k'}$, alors l'isomorphisme canonique :
$$j^{\star}\pi_{\star}{\mathcal L}\longrightarrow\tilde{\pi}_{\star}\tilde{\mathcal L}$$
d\'efinit un isomorphisme isom\'etrique de fibr\'es vectoriels hermitiens sur la base $\Spec({\mathcal O}_{k'})$ :
$$j^{\star}\pi_{\star}\overline{\mathcal L}\longrightarrow\tilde{\pi}_{\star}\overline{\tilde{\mathcal L}}\;.$$
\item[{\rm (v)}]{\em Pente de $\pi_{\star}\overline{\mathcal L}$}. Pour tout $MB$-mod\`ele comme en {\rm (i)}
on a $\pi_{\star}\overline{\mathcal L}$ semi-stable et :
$$\frac{\widehat{\degr}\pi_{\star}\overline{\mathcal
L}}{[K:\mathbb{Q}]\rho(L)}=-\frac{1}{2} h_{F}(A)+\frac{1}{4}\log\left(\frac{\rho(L)}{(2\pi)^g}\right).$$
\item[{\rm (vi)}] {\em Points base}. Pour tout $MB$-mod\`ele comme en {\rm (i)}, et tout $n\in\mathbb{N}^{\star}$, soit ${\mathcal
A}^{[n]}$ le plus petit sch\'ema en groupes ouvert  de $\mathcal A$ contenant $K({\mathcal L}_{\overline{\mathbb{Q}}}^{\otimes n})$. 
Si $n$ est pair et si la cl\^oture de $K({\mathcal L}_{\overline{\mathbb{Q}}}^{\otimes n})$ dans $\mathcal A$ est finie sur
$\Spec({\mathcal O}_k)$, alors les sections globales ${\rm H}^0({\mathcal A},{\mathcal L}^{\otimes n})$ engendrent ${\mathcal
L}^{\otimes n}$ sur ${\mathcal A}^{[n]}$.
\end{itemize}
\end{thm}
\begin{proof} Les preuves de {\rm (i)}-{\rm (v)} sont donn\'ees dans 
\cite{Bost2}, \S~4.~3.~2. Le {\rm (vi)} vient de \cite{Moret}, VI.~3.4 et
VI.~2.~2.
\end{proof}
\begin{rem}
On remarque que si $(\pi\colon{\mathcal A}\longrightarrow \Spec({\mathcal O}_k), i, {\overline{\mathcal L}},
\varphi,(\varepsilon_P)_{P\in F})$ est un $MB$-mod\`ele, alors $(\pi\colon{\mathcal A}\longrightarrow \Spec({\mathcal O}_{k'}), i, {\overline{\mathcal L}^{\otimes r^2}},
\varphi,(\varepsilon_P)_{P\in F})$ est aussi un $MB$-mod\`ele pour une extension $k'$ de $k$.
\end{rem}

\section{Hauteur d'un point par la formule clef}\label{formuleclef}

Soit $(A,L)$ une vari\'et\'e ab\'elienne principalement polaris\'ee et d\'efinie sur un corps de nombres $k$ et soit $F=A[4]$ le sous-groupe de 4-torsion de $A$. On se donne un MB-mod\`ele associ\'e \`a $(A,L,F)$ comme dans la partie pr\'ec\'edente $(\pi\colon{\mathcal A}\longrightarrow \Spec({\mathcal O}_K), i, {\overline{\mathcal L}},
\varphi,(\varepsilon_P)_{P\in F})$. On rappelle ici quelques notations de l'article \cite{Paz2}. 

Soit $\Gamma$ un syst\`eme de repr\'esentants de $A_{r^2}/A_r$, le quotient des points de $r^2$-torsion par le groupe des points de $r$-torsion. Notons $\mathcal{F}=(\pi_{\star}{\mathcal
L})^\Gamma$ la somme directe de $r^{2g}$ copies de $\pi_{\star}{\mathcal
L}$, index\'ee par $\Gamma$. Le sous-sch\'ema ${\mathcal B}_{\mathcal F}$ des points bases du syst\`eme lin\'eaire $\mathcal F$ de sections de $\mathcal{L}^{\otimes r^2}$ est d\'efini comme le sous-sch\'ema ferm\'e de $\mathcal A$ dont le faisceau d'id\'eaux ${I}_{{\mathcal B}_{\mathcal F}}$ est tel que l'image de l'application canonique $\pi^{\star}{\mathcal F}\longrightarrow {\mathcal{L}^{\otimes r^2}}$
soit ${I}_{{\mathcal B}_{\mathcal F}}.{\mathcal{L}^{\otimes r^2}}$.

Comme ${\mathcal B}_{\mathcal F}$ ne coupe pas la fibre g\'en\'erique ${\mathcal A}_k$, pour toute section $P$ de $\pi$, le sous-sch\'ema $P^{\star}{\mathcal B}_{\mathcal F}$ de $\Spec({\mathcal O}_k)$ est un diviseur. On notera~:

\begin{defin}\label{beta1}
\begin{equation}\label{betafini}
P^{\star}{\mathcal B}_{\mathcal F}=\sum_{\begin{array}{c}{\mathfrak p}\, {\rm premier}\\ {\rm de}\, {\mathcal
O}_k,\\ {\mathfrak p}\nmid\infty\end{array}}\beta_{\mathfrak p}({\mathcal L}^{\otimes r^2},{\mathcal F},P){\mathfrak p}\;.
\end{equation}
\end{defin}

\noindent Dans cette d\'efinition, les $\beta_{\mathfrak p}$ sont des entiers positifs presque tous nuls. Le fibr\'e $\mathcal{F}$ varie avec le fibr\'e $\mathcal{L}$.

\begin{defin}
\label{beta2}
Soit $(A,L)$ une vari\'et\'e ab\'elienne principalement polaris\'ee sur un corps de nombres $k$, de dimension $g$, avec $L$ sym\'etrique et ample. Soit $P$ un point de ${\mathcal A}({\mathcal O}_k)$ et $\sigma\colon k\hookrightarrow\mathbb{C}$ un plongement complexe. Soit $\tau_{\sigma}$ un point de l'espace de Siegel ${\mathfrak S}_g$ tel que~:
\begin{equation}
\label{choixtau}
{\mathcal A}_{\sigma}(\mathbb{C})\simeq\mathbb{C}^g/(\mathbb{Z}^g+\tau_{\sigma}\mathbb{Z}^g)
\end{equation}
comme vari\'et\'es ab\'eliennes principalement polaris\'ees, et soit $z\in\mathbb{C}^g$ tel que
$[z]\in\mathbb{C}^g/(\mathbb{Z}^g+\tau_{\sigma}\mathbb{Z}^g)$ soit l'image de $P_{\sigma}$ par l'application~(\ref{choixtau}). Alors on pose~:
\begin{equation}
\label{betainfini}
\beta_{\sigma}(\overline{\mathcal L}^{\otimes r^2},\overline{\mathcal F},P)
=-\frac{1}{2}\log\left(2^{\frac{g}{2}}\sum_{e\in{\mathcal
Z}_r(\tau_{\sigma})}\Vert\theta\Vert^2(rz+e,\tau_{\sigma})\right),
\end{equation}
o\`u on a not\'e ${\mathcal Z}_r(\tau_{\sigma})$ l'ensemble
$\frac{1}{r}(\mathbb{Z}^g+\tau_{\sigma}\mathbb{Z}^g)/(\mathbb{Z}^g+\tau_{\sigma}\mathbb{Z}^g)$ et o\`u on utilise $$\Vert\theta\Vert(z,\tau)=\mathrm{det}(\Ima\tau)^{\frac{1}{4}}\;\mathrm{exp}(-\pi\,^t\Ima z(\Ima \tau)^{-1}\Ima z)\;\vert\theta(z,\tau)\vert.$$ 
\end{defin}

\noindent Dans cette derni\`ere d\'efinition, la d\'ependance en $\mathcal{F}$ est cach\'ee dans le choix des fonctions th\^eta, on trouvera plus de d\'etails dans le paragraphe 4 de \cite{Paz2}. La d\'efinition $(\ref{betainfini})$ est la d\'efinition des $\beta_\sigma$ de \cite{Bost2}. 

\begin{defin}\label{beta3}
Pour $v$ une place finie associ\'ee \`a un premier $\mathfrak{p}$ de $\mathcal{O}_k$, on notera $$\beta_{L^{\otimes r^2},v}(P)=\beta_{\mathfrak p}({\mathcal L}^{\otimes r^2},{\mathcal F},P)\log \mathrm{N}_{k/\mathbb{Q}}({\mathfrak p}).$$
Pour $v$ une place archim\'edienne associ\'ee \`a un plongement $\sigma$, on notera $$\beta_{L^{\otimes r^2},v}(P)=\beta_{\sigma}(\overline{\mathcal L}^{\otimes r^2},\overline{\mathcal F},P).$$
\end{defin}

\begin{prop}\label{clef}
Soit $(A,L)$ une vari\'et\'e ab\'elienne principalement polaris\'ee sur un corps de nombres $k$ de degr\'e $d$, de dimension $g$, semi-stable, avec $L$ sym\'etrique et ample. Pour toute section $P$ de $\pi\colon {\mathcal A}\longrightarrow\Spec({\mathcal O}_k)$, on a :
$$\widehat{h}_{A,L^{\otimes r^2}}(P)  =  h_{A,L^{\otimes r^2}}(P)-\frac{1}{2} \hFplus(A)+\frac{1}{d}\sum_{v\in{M_k}} d_v \beta_{L^{\otimes r^2},v}(P)$$
o\`u $\beta_{L^{\otimes r^2},v}$  est d\'efini par (\ref{beta3}).
\end{prop}

\begin{proof}
La preuve de cette proposition est donn\'ee dans \cite{Paz2} au lemme 5.2, modulo le passage \`a la hauteur de Faltings positive qui simplifie ici la formule en vertu de $\hFplus(A)=\hF(A)+\frac{g}{2}\log(2\pi^2)$.
\end{proof}

\section{D\'ecomposition de la hauteur de Faltings d'une jacobienne hyperelliptique}\label{hyper}

On s'int\'eresse dans cette partie aux jacobiennes de courbes hyperelliptiques. On donne une formule explicite pour la hauteur de Faltings d'une telle jacobienne $A$, g\'en\'eralisant ainsi la formule propos\'ee par Ueno dans \cite{Ueno} pour la dimension 2. Dans toute la suite $k$ d\'esigne un corps de nombres, $\mathcal{O}_{k}$ son anneau d'entiers et on note $S=\Spec(\mathcal{O}_{k})$. Soient $C/k$ une courbe hyperelliptique et $A=J_C$ sa jacobienne.

\subsection*{Equations de Weierstrass et discriminants}\label{Weierstrass}

Soient $g\geq 2$ et $k$ un corps de nombres. Soit $v\in{M_k}$. On note $k_v$ le compl\'et\'e de $k$ en la place $v$, vu comme corps de fractions de l'anneau de valuation discr\`ete $\mathcal{O}_v$. Nous allons travailler avec les \'equations de Weierstrass des courbes hyperelliptiques et leur discriminant associ\'e. Pour obtenir une meilleure formule close, nous allons utiliser les travaux de Lockhart \cite{Lock} aux places archim\'ediennes, qui fournissent via la th\'eorie de Mumford une contribution explicit\'ee en termes de fonctions th\^eta. Pour les places finies nous utiliserons les travaux de Kausz \cite{Kausz}, Maugeais \cite{Maug} et Liu \cite{Liu2, Liu3} qui impliqueront le discriminant minimal des mod\`eles de Weierstrass.


%





\begin{defin}\label{Lockhart}(Lockhart)
Un mod\`ele de Weierstrass ($E$)
d'une courbe hyperelliptique $C$ sur $k_v$ est une \'equation du type
$$(E)\; :\;\,  y^{2}+Q(x)\, y=P(x),$$
o\`u $P$ et $Q$ sont des polyn\^omes \`a coefficients dans $k_v$ tels que $\deg Q
\leq g$ et $\deg P = 2g+1$. On choisit $P$ unitaire. Une telle \'equation est unique modulo les changements de variables

\begin{center}
 
$(*)\left\{ 
\begin{tabular}{l}
$x=u^{g}x'+s $ \\ 
\\

$y=u^{2g+1}y'+t(x')$ \\ 

\end{tabular}
\right.  $
\end{center}

\noindent o\`u $u\in{k_v^{*}}$, $s\in{k_v}$ et $t$ est un polyn\^ome \`a coefficients dans
$k_v$ de degr\'e inf\'erieur ou \'egal \`a $g$ (voir proposition 1.2 page 730 de \cite{Lock}).

Le discriminant de ($E$) est d\'efini par $\Delta_{Lock}(E) = 2^{4g}\disc\left(P(x)+\frac{1}{4}\,Q(x)^{2}\right).$ Le discriminant minimal de Lockhart $\Delta_{Lock}(C)$ de $C/k_v$ sera le discriminant de valuation minimale parmi tous les mod\`eles $(E)$ relativement aux changements de variables $(*)$.

\end{defin}

\begin{defin}\label{Liu}(Liu)
Soient $P,Q\in \mathcal{O}_v[x]$ des polyn\^omes avec $\deg P \le 2g+2$ and $\deg Q\le
g+1$ tels que
\begin{equation}
\label{eqg}
(\mathcal{E}):  y^2 + Q(x)y = P(x)
\end{equation}
est une \emph{\'equation enti\`ere pour $C$} o\`u $C/k_v$ est hyperelliptique de genre $g$. 
On note $\Delta_{Liu}(\mathcal E)=2^{-(4g+4)}\disc_{2g+2}(4P+Q^2)$ le discriminant de (\ref{eqg}).
Si on \'ecrit de plus
\begin{equation*}
  P = p_{2g+2} x^{2g+2} + \cdots + p_0 \quad\text{et}\quad
Q = q_{g+1} x^{g+1} + \cdots +q_0
\end{equation*}
o\`u $p_{2g+2}$ et $q_{g+1}$ sont des \'el\'ements (\'eventuellement nuls) de $\mathcal{O}_v$. On a unicit\'e de ce mod\`ele modulo les changements de variables (voir \cite{Liu3}, corollaire 4.33 page 296)

\begin{center}
 
$(**)\left\{ 
\begin{tabular}{l}
$\displaystyle{x=\frac{ax'+b}{cx'+d} }$ \\ 
\\

$\displaystyle{y=\frac{H(x')+e y'}{(cx'+d)^{g+1}}}$ \\ 

\end{tabular}
\right.  $
\end{center}
o\`u $\left(\begin{array}{cc}a & b \\ c & d\end{array}\right)\in{GL_2(K)}$, $e\in{K^*}$, $H\in{K[x']}$ avec $\deg H\leq g+1$.

Le discriminant minimal de Liu $\Delta_{Liu}(C)$ de $C$ sera le discriminant de valuation minimale parmi tous les mod\`eles $(\mathcal  E)$ relativement aux changements de variables $(**)$, voir \cite{Liu3} 1.26 page 464 and 1.9 page 480. 

\end{defin}

Lorsqu'on passe d'une \'equation $E$ \`a une \'equation $E'$ par le changement de coordonn\'ees ($*$) on obtient la
relation $\Delta_{E}=u^{\displaystyle{4g(2g+1)}}\Delta_{E'}$.

\subsection*{Formes diff\'erentielles}

Soit $C$ une courbe hyperelliptique donn\'ee par un mod\`ele de
Weierstrass $E$. On peut alors exhiber une base de
$H^{0}(C,\Omega^{1}_{C/k})$, donn\'ee par les formes :
$$\omega_{i}=\frac{x^{i-1}dx}{2y+Q(x)}, \; 1\leq i \leq g .$$
Consid\'erons alors la $g$-forme  $\alpha=\omega_{1}\wedge ... \wedge
\omega_{g}$. On v\'erifie que lorsqu'on change de coordonn\'ees en utilisant ($*$) dans le
mod\`ele de Weierstrass pour passer de $E$ \`a $E'$ on obtient la relation
$\alpha=u^{-g^{2}}\alpha'$. Ceci va nous permettre de trouver une
$g$-forme diff\'erentielle $\eta$ ne d\'ependant pas du mod\`ele de Weierstrass.
 On va chercher $\eta$ de la forme :
$$\eta=\Delta_{E}^{\displaystyle{a}}(\omega_{1}\wedge ... \wedge
\omega_{g})^{\displaystyle{b}},\,\, \mathrm{avec}\;\, a,b\in{\mathbb{Z}}.$$
Un changement de mod\`ele de $E$ vers $E'$ de la forme ($*$) conduit alors \`a :
$$\eta =u^{\displaystyle{4g(2g+1)a-g^{2}b}}\eta'.$$
 On fait le choix $a=g$ et
$b=4(2g+1)$ et on obtient :

\begin{prop}
La $g$-forme $\eta=\Delta_{E}^{\displaystyle{g}}\,(\omega_{1}\wedge ... \wedge
\omega_{g})^{\displaystyle{\otimes 4(2g+1)}}$ est ind\'ependante du mod\`ele de
Weierstrass $E$ choisi. Elle est de plus bien d\'efinie sur les mod\`eles $(\mathcal{E})$ et est aussi ind\'ependante du mod\`ele $\mathcal{E}$ choisi.
\end{prop}

\begin{proof}
L'invariance plus g\'en\'erale par le changement de variables $(**)$ est un calcul direct. Cela avait d\'ej\`a \'et\'e remarqu\'e dans \cite{Kausz} proposition 2.2 page 43.
\end{proof}

On va donc utiliser la section $\eta$ pour calculer la
hauteur de Faltings exprim\'ee comme en (\ref{hauteur de Faltings relative}). On a besoin pour cela du r\'esultat de la proposition suivante. L'invariance par les deux classes de changements de variables nous permettra de tirer parti du meilleur mod\`ele aux places finies (suivant \cite{Kausz}) et d'un mod\`ele pratique au regard des fonctions th\^eta aux places archim\'ediennes (suivant \cite{Lock}).

On commence par d\'efinir :

\begin{defin}

Soient $k$ un corps de nombres et $C/k$ une courbe alg\'ebrique lisse \`a r\'eduction semi-stable d\'efinie sur $k$ et de genre $g$. On note $p:C\rightarrow S$ un mod\`ele entier sur $S=\Spec(\mathcal{O}_{k})$ semi-stable de la courbe $C$. On appelle hauteur de Faltings de $C/k$ la quantit\'e :
$$\hFplus(C)=\frac{1}{[k:\mathbb{Q}]}\degrar (\det p_{*}\omega_{C/S}),$$
o\`u le choix de m\'etriques hermitiennes est $\Vert\alpha\Vert_v^2=\frac{i^{g^2}}{(2\pi)^{2g}}\int\alpha\wedge\overline{\alpha}$.
\end{defin}
On remarquera qu'il s'agit bien de la hauteur stable car on travaille sur un mod\`ele semi-stable de $C$.

\begin{prop}
Soient $k$ un corps de nombres et $C/k$ une
courbe alg\'ebrique lisse \`a r\'eduction semi-stable. On a alors :
$$\hFplus(J_C)=\hFplus(C). $$
\end{prop}

\begin{proof}
Notons $S=\Spec{\mathcal{O}_{k}}$ et soit $p:C\longrightarrow S$ un mod\`ele
entier semi-stable de la courbe $C$, de section neutre $\varepsilon$. On consid\`ere $A=\Pic^{0}_{C/S}$. On a alors :
$$\Lie(A)\simeq R^{1}p_{*}O_{C}.$$
De plus par dualit\'e de Grothendieck (on pourra consulter le paragraphe 6.4.3 page 243 de \cite{Liu3}):
$$(R^{1}p_{*}O_{C})^{\vee}\simeq p_{*}\omega_{C/S}. $$
On calcule alors :
$$\varepsilon^{*}\Omega^{1}_{A/S}\simeq \Lie(A)^{\vee}\simeq
p_{*}\omega_{C/S}, $$
d'o\`u :
$$\varepsilon^{*}\Omega^{g}_{A/S}\simeq \det p_{*}\omega_{C/S},$$ et cet isomorphisme est une isom\'etrie d'apr\`es le 4.15 de l'expos\'e II de \cite{SPA}. Il suffit alors de prendre le degr\'e d'Arakelov de chaque côt\'e pour obtenir la proposition.
\end{proof}

\subsection*{Partie non archim\'edienne}

On garde les notations des paragraphes pr\'ec\'edents. La section $\eta$ correspond au choix de section $\Lambda$ fait dans \cite{Kausz}. Dans cet article, Kausz analyse les contributions en chaque place finie, ce qui permet d'obtenir la proposition suivante :

\begin{prop}\label{hyperelliptique finie}
La section $\eta$ s'\'etend en une section globale enti\`ere sur le mod\`ele de la courbe $C$. Il existe des entiers naturels $e_v$ tels que la somme des contributions aux places finies s'exprime par :

\begin{center}
$\begin{tabular}{l}
$\log\Big(\Card(\det p_{*}\omega_{C/S}/\eta\, \mathcal{O}_{k})\Big)= $ \\
\\
$\displaystyle{g\,\log \Nk(\Delta_{\mathrm{Liu}}(C))-\sum_{v|\Delta_{\mathrm{Liu}}(C)}d_v(8g+4)e_{v}\log\Nk(v).}$\\

\end{tabular}$
\end{center}
\end{prop}

\begin{proof}
La section s'\'etend sur le mod\`ele entier d'apr\`es le th\'eor\`eme 3.1 page 44 de \cite{Kausz}, on a ainsi $\mathrm{ord}_v(\eta)\geq 0$.
En reprenant les calculs \`a l'\'equation $(1)$ de la preuve de la proposition 5.5 page 57 de \textit{loc.cit.}, on peut d\'eduire que pour toute place finie $v$ il existe un entier $e_{v}\in{\mathbb{N}}$ d\'ependant du mod\`ele hyperelliptique et tel que :
$$\ordv(\eta)= g\ordv(\Delta_{\mathcal{E}})-(8g+4)e_{v}.$$
 Notons qu'en combinant avec l'article de Maugeais \cite{Maug}, th\'eor\`eme 1.1 (voir aussi \cite{Maug2} pour une version corrigée) on peut supprimer l'hypoth\`ese de bonne r\'eduction en $2$ faite dans \cite{Kausz}. 
\end{proof}

\subsection*{Partie archim\'edienne}

On se base ici sur les travaux de Lockhart \cite{Lock}, dont les
calculs s'appuient en bonne partie sur \cite{Mum2}, lequel n\'ecessite un mod\`ele hyperelliptique de degr\'e impair. Rappelons ici les notations de l'introduction. Tout d'abord pour $m\in{\frac{1}{2}\mathbb{Z}^{2g}}$ on pose :
$$\varphi_{m}(\tau)=\theta_{m}(0,\tau)^{8},$$
o\`u $\theta_{m}(z,\tau)$ est la fonction th\^eta de caract\'eristique $m$ associ\'ee au r\'eseau de
dimension $g$ dont la d\'efinition est rappel\'ee en (\ref{fonction theta}). Si $S$ est un sous-ensemble de $\{1,2,...,2g+1\}$ on d\'efinit alors
$\displaystyle{m_{S}=\sum_{i\in{S}}m_{i}\in{\frac{1}{2}\mathbb{Z}^{2g}}}$ avec :

\begin{center}
\begin{tabular}{llll}
$m_{2i-1}$ & $=$ & $\left[ \begin{array}{ccccccc}
^{t}(0 & ... & 0 & \frac{1}{2} & 0 & ... & 0) \\
^{t}(\frac{1}{2} & ... & \frac{1}{2} & 0 & 0 & ... & 0) \\ 	
\end{array}\right],$ & $1\leq i \leq g+1 \, ,$ \\
\\

\end{tabular}
\end{center}

\begin{center}
\begin{tabular}{llll}
$m_{2i}$ & $=$ & $\left[ \begin{array}{ccccccc}
^{t}(0 & ... & 0 & \frac{1}{2} & 0 & ... & 0) \\
^{t}(\frac{1}{2} & ... & \frac{1}{2} & \frac{1}{2} & 0 & ... & 0) \\ 	
\end{array}\right],$ & $1\leq i \leq g \, ,$ \\
\\

\end{tabular}
\end{center}

\noindent o\`u le coefficient non nul de la premi\`ere ligne est en i-\`eme
position. Soit alors $\mathcal{T}$ la collection des sous-ensembles de
$\{1,...,2g+1\}$ de cardinal $g+1$. Soit $U=\{1,3,...,2g+1\}$ et
notons $\circ$ l'op\'erateur de diff\'erence sym\'etrique. On d\'efinit alors
:

\begin{equation}\label{phitheta2}
\varphi(\tau)=\prod_{T\in{\mathcal{T}}}\varphi_{m_{T\circ U}}(\tau).
\end{equation}

\noindent On pose enfin $l=\binom{2g+1}{g+1}$ et $n=\binom{2g}{g+1}=\frac{g}{2g+1}l$. On a alors la
proposition suivante dont la preuve figure dans \cite{Lock} :

\begin{prop}(Lockhart)
Soient $C/\mathbb{C}$ une courbe hyperelliptique et $P\in{C(\mathbb{C})}$. Soit $E$ un
mod\`ele de Weierstrass de $(C,P)$. On uniformise $J_C(\mathbb{C})\simeq \mathbb{C}^{g}/\Lambda_{E}$ avec le r\'eseau $\Lambda_{E}=\Omega_{1}\mathbb{Z}^{g}+\Omega_{2}\mathbb{Z}^{g}$ et $\tau_{E}=\Omega_{1}^{-1}\Omega_{2}$. Soit $V(\Lambda_{E})$ le covolume du
r\'eseau $\Lambda_{E}$ dans $\mathbb{C}^{g}$. Alors la quantit\'e
$|\Delta_{E}|V(\Lambda_{E})^{4+\frac{2}{g}}$ ne d\'epend pas du choix de
$E$ et on a :
$$|\Delta_{E}|V(\Lambda_{E})^{4+\frac{2}{g}}=2^{4g}\pi^{8g+4}\Big(|\varphi(\tau_{E})|\det(\Ima
\tau_{E})^{2l}\Big)^{\frac{1}{n}}. $$
\end{prop}

Nous allons ainsi pouvoir calculer la contribution archim\'edienne \`a la
hauteur de Faltings positive (voir \cite{Dej} pour un calcul similaire) :

\begin{prop}\label{hyperelliptique infinie}
Soient $C/k$ une courbe de genre g donn\'ee dans un mod\`ele hyperelliptique
$E$ de discriminant $\Delta_{E}$ et
$\omega_{i}=\frac{x^{i-1}dx}{2y+q(x)},\, i\in\{1,... \,,g\},$ une base des
formes diff\'erentielles sur $C$.
Soit $v\in{M_{k}^{\infty}}$.
Soit $\eta=\Delta_{E}^{g}(\omega_{1}\wedge
... \wedge\omega_{g})^{\otimes 8g+4}$. Alors on a :
$$\log||\eta||_{v}=(8g+4)\log\left(2^{-\frac{g}{4g+2}}|\varphi(\tau_{E})|^{\frac{1}{4l}}\det(\Ima
\tau_{E})^{\frac{1}{2}}\right).$$ 

\end{prop}

\begin{proof}

Il suffit de calculer :
\\

\begin{tabular}{lll}

$||\eta||_{v}^{2}$ & $=$ & $\displaystyle{|\Delta_{E}|^{2g}(||\omega_{1}\wedge...\wedge \omega_{g}||_{v}^{2})^{8g+4}}$\\
\\

$$ & $=$ & $\displaystyle{|\Delta_{E}|^{2g}\frac{1}{(2\pi)^{2g(8g+4)}}\left(\int_{A_{v}(\mathbb{C})}i^{g^2}\omega_{1}\wedge...\wedge \omega_{g}\wedge\overline{\omega_{1}}\wedge...\wedge \overline{\omega_{g}}\right)^{8g+4}}$\\
\\

$$ & $=$ & $\displaystyle{|\Delta_{E}|^{2g}\frac{1}{(2\pi)^{2g(8g+4)}}(2^gV(\Lambda_{E}))^{8g+4}}$\\
\\

$$ & $=$ & $\displaystyle{\frac{2^{g(8g+4)}}{(2\pi)^{2g(8g+4)}}|\Delta_{E}|^{2g}V(\Lambda_{E})^{8g+4}}$\\
\\

$$ & $=$ & $\displaystyle{\frac{2^{g(8g+4)}}{(2\pi)^{2g(8g+4)}}2^{8g^2}\pi^{2g(8g+4)}(|\varphi(\tau_{E})|\det(\Ima
\tau_{E})^{2l})^{\frac{2g}{n}}}$\\
\\

$$ & $=$ & $\displaystyle{2^{-4g}(|\varphi(\tau_{E})|\det(\Ima
\tau_{E})^{2l})^{\frac{2g}{n}}},$\\
\\

\end{tabular}

\noindent donc $$||\eta||_{v}=\displaystyle{2^{-2g}|\varphi(\tau_{E})|^{\frac{g}{n}}\det(\Ima
\tau_{E})^{\frac{2lg}{n}}}.$$
La formule est \'etablie en observant que $\frac{g}{n}=\frac{8g+4}{4l}$ et $\frac{2lg}{n}=\frac{8g+4}{2}$. On obtient lorsque $g=1$, comme $\varphi(\tau)=2^8\Delta(\tau)$ (voir \cite{Lock} page 740), l'expression $\Vert \eta\Vert_v=2^6\Delta(\tau)(\Ima\tau)^6$.

\end{proof}

\subsubsection*{Preuve du th\'eor\`eme \ref{hauteur faltings jacobienne}}

\begin{proof}

On d\'eduit la preuve du th\'eor\`eme en mettant bout \`a bout les calculs des deux paragraphes pr\'ec\'edents, rassembl\'es dans les propositions \ref{hyperelliptique finie} et \ref{hyperelliptique infinie}. On peut choisir le discriminant minimal dans la formule finale en minimisant le mod\`ele place par place. On d\'efinit donc le discriminant $\Delta_{\mathrm{min}}$ comme le produit des discriminants minimaux locaux. On a donc par d\'efinition que $\mathrm{ord}_v(\Delta_{\mathrm{min}/k})=\mathrm{ord}_v(\Delta_{\mathrm{Liu}}(C/k_v))$ pour chaque place $v$ finie de $k$. Le calcul donne naturellement l'expression de $(8g+4)\hFplus(\Jac(C))$.

\end{proof}

\section{Calculs explicites en dimension 1}\label{dim1}

On montre dans les paragraphes qui suivent que la formule de hauteur d'une courbe elliptique est bien la m\^eme dans toutes ces diff\'erentes approches.

\subsection{Formule hyperelliptique}

La sp\'ecialisation au cas du genre $g=1$  du th\'eor\`eme \ref{hauteur faltings jacobienne} fournit $l=\binom{2g+1}{g+1}=3$. Soit $E$ une courbe elliptique d\'efinie sur un corps de nombres $k$. Pour toute place archim\'edienne $v$, on note $\tau_v$ un nombre complexe tel que ${E}_{v}(\mathbb{C})\simeq\mathbb{C}/(\mathbb{Z}+\tau_{v}\mathbb{Z})$ et $\Delta_{\mathrm{min}}$ le discriminant minimal de $E/k$. Alors $e_v= 0$ et $12 f_v= \mathrm{ord}_v(\Delta_{\mathrm{min}})$, donc

\begin{center}
$\begin{tabular}{ll}
$\displaystyle{d\cdot\hFplus(E)=}$ & $\displaystyle{ \sum_{v\vert\Delta_{\mathrm{min}}}d_v \frac{1}{12}\mathrm{ord}_v(\Delta_{\mathrm{min}})\log\Nk(v)}$ \\
\\
&$\displaystyle{-\sum_{v\in{M_k^{\infty}}}d_v \log\left(2^{-\frac{2}{12}}|\varphi(\tau_{v})|^{\frac{1}{4\cdot 3}}\det(\Ima
\tau_{v})^{\frac{1}{2}}\right),}$\\
\end{tabular}$
\end{center}

\noindent de plus il y a trois sous-ensembles de cardinal $2$ dans $\{1,2,3\}$, qui fournissent apr\`es un rapide calcul $\varphi(\tau)=\theta_2^8\theta_3^8\theta_4^8=2^{8}\Delta(\tau)$. On obtient bien

$$12d\cdot\hFplus(E)=\!\sum_{v\vert\Delta_{\mathrm{min}}}\!d_v \mathrm{ord}_v(\Delta_{\mathrm{min}})\log\Nk(v)-\!\sum_{v\in{M_k^{\infty}}}\!d_v \log\left(|2^6\Delta(\tau_{v})|\det(\Ima
\tau_{v})^{6}\right).$$

\subsection{Formule d'Autissier}

La sp\'ecialisation au cas des courbes elliptiques de la formule d'Autissier donn\'ee au th\'eor\`eme \ref{Autissier} fournit :

$$\hFplus(E)=2\hat{h}_L(\Theta)+\frac{2}{d}\sum_{v\in{M_k^{\infty}}}d_v I(E_v,\lambda_v).$$

\noindent On a $\hat{h}_L(\Theta)=0$  (voir la partie de d\'efinition page 1453 de \cite{Aut}) car le diviseur est de torsion dans le cas elliptique. De plus par la proposition 2.1 page 1452 de \cite{Aut} il vient 
$$I(E,\lambda)=-\frac{1}{24}\log\left(\vert \Delta(\tau)\vert (2\Ima\tau)^6\right).$$

On obtient donc bien la m\^eme formule que celle du th\'eor\`eme \ref{elliptique} pour le cas de bonne r\'eduction partout. Dans le cas g\'en\'eral, pour les contributions locales aux places finies, le r\'esultat est donn\'e dans le th\'eor\`eme 7 de Faltings \cite{Fal}.

\bibliographystyle{amsalpha}

\vfill
{\flushleft
Fabien Pazuki\\
Th{\'e}orie des nombres, IMB, Universit{\'e} de Bordeaux\\
351, cours de la Lib\'eration, 33 405 Talence Cedex, France\\
e-mail : fabien.pazuki@math.u-bordeaux.fr
{\flushleft
et\\
Department of Mathematical Sciences, University of Copenhagen\\
Universitetsparken 5, DK-2100 Copenhagen Ø, Denmark \\
}

\end{document}